\documentclass[12pt]{amsart}

\usepackage{amsmath}
\usepackage{amssymb}  
\usepackage{latexsym} 
\usepackage{comment}
\usepackage{url}
\usepackage{fullpage,url,amssymb,amsmath,amsthm,amsfonts,mathrsfs}
\usepackage[usenames,dvipsnames]{color}
\usepackage[pagebackref = true, colorlinks = true, linkcolor = blue, citecolor = Green]{hyperref}
\usepackage[alphabetic,lite]{amsrefs}
\usepackage{enumitem}
\usepackage{amscd}   
\usepackage[all, cmtip]{xy} 
\usepackage{xfrac}
\usepackage[T1]{fontenc}
\usepackage{subfiles}
\usepackage{colonequals}

\usepackage[all]{xy}
\usepackage{fullpage}

\usepackage{color} 

\def\act#1#2%
  {\mathop{}%
   \mathopen{\vphantom{#2}}^{#1}%
   \kern-3\scriptspace%
   #2}

\newcommand{\F}{{\mathbb F}}

\newtheorem{Theorem}{Theorem}[section]
\newtheorem{Lemma}[Theorem]{Lemma}

\newtheorem{Definition}[Theorem]{Definition}
\newtheorem{Example}[Theorem]{Example}

\begin{document}

\title{Greedy capsets}

\author{Oliver Dawson}
\address{School of Mathematics and Statistics, University of Canterbury, Private Bag 4800, Christchurch 8140, New Zealand}
\email{oda35@uclive.ac.nz}

\author{Oleg Shuvaev}
\address{School of Mathematics and Statistics, University of Canterbury, Private Bag 4800, Christchurch 8140, New Zealand}
\email{osh22@uclive.ac.nz}

\author{Jos\'e Felipe Voloch}
\address{School of Mathematics and Statistics, University of Canterbury, Private Bag 4800, Christchurch 8140, New Zealand}
\email{felipe.voloch@canterbury.ac.nz}
\urladdr{http://www.math.canterbury.ac.nz/\~{}f.voloch}


\begin{abstract}
A capset is a subset $C \subset \F_3^n$ with no three points on a line. We characterise the capsets produced by successively removing points from the ambient space such that the removed point has the maximum number of lines contained in the set of remaining points and passing through it until the set of remaining points contains no lines.
\end{abstract}

\maketitle
%

\section{Introduction}

A capset is a subset $C \subset \F_3^n$ with no three points on a line. A question that has attracted a lot of interest (see e.g. \cite{Tao},\cite{Matrix}) is the following: What is the largest size of a capset in $\F_3^n$ as a function of $n$? If we denote this value by $a(n)$ then it can be shown (see \cite[Proposition 2.2]{Tyrrell}) that $c=\lim a(n)^{1/n}$ exists. Moreover, it is known that $2.2202 \le c$ (\cite{Nature}) and $c \le 2.756$ (\cite{Jordan}). The exact value of $a(n)$ is known for $n \le 6$.

The best current lower bound for $a(8)$, namely $a(8) \ge 512$, was recently obtained via Machine Learning \cite{Nature}. The authors of that paper trained a Large Language Model to produce algorithms that generate capsets. They only considered algorithms that started from the empty set and successively added a point to the set as long as adding it did not violate the capset property. The algorithms varied on the different ways of selecting the next point to be added. The present paper arose from the idea that, instead, one could start with the whole space and successively remove points until one obtains a capset. There is a natural way for prioritizing the choice of the next removed point in this latter procedure, namely taking a point such that the number of lines contained in the set of remaining points and passing through the given point is maximum. We programmed this alternative approach (code available in \cite{git}) and, contrary to our expectations, it did not lead to large capsets. Indeed, it always gave capsets of size $2^n$ with a very specific structure (described in Definition \ref{def:greedy}). The purpose of this paper is to explain this phenomenon. This is achieved in Theorem \ref{thm:main}. Incidentally, we tried a few variants of this latter procedure (see \cite{git}) and these did not produce large capsets either.

\subsection*{Acknowledgements}
The third author was supported by the Ministry for Business, Innvovation and Employment in New Zealand and the Marsden Fund administered by the Royal
Society of New Zealand.

\section{Affine spaces}

We will always be working in the ambient space $\F_3^n$ for some integer $n \ge 1$. This is a vector space over the field $\F_3$ of three elements. So, we can consider (vector) subspaces of $\F_3^n$ and any of those has a dimension $d \ge 0$ and, because our field of scalars is $\F_3$, a subspace of dimension $d$ has $3^d$ elements. 

We can also talk about affine subspaces of $\F_3^n$, which are translations of vector subspaces and, therefore, we can also consider their dimension and, again, an affine subspace of dimension $d$ has $3^d$ elements. We will often refer to hyperplanes in an affine space and this simply means a subspace of codimension 1 (i.e. dimension 1 less than the ambient space). A hyperplane can be described as the set of solutions $(x_1,\ldots,x_n)$, $x_i \in \F_3$ of a single equation $H:\sum_i a_i x _i + b = 0$ where $a_i, b, \in \F_3$ and the $a_i$ are not all zero. The hyperplanes given by $\sum_i a_i x _i + c, c \ne b$ are the hyperplanes parallel to $H$.

A line is an affine subspace of dimension $1$, so a line has three points. It is not hard to show that three distinct points $P,Q,R$ form a line if and only if $P+Q+R=0$.
We say that a set $S$ generates an affine subspace $V \subset \F_3^n$ if $V$ is the smallest subspace containing $S$. Finally, we notice that a subgroup of $\F_3^n$ is a vector subspace since our field of scalars is just $\F_3 = \{0,1,2\}$.

\begin{Lemma}
\label{lem:grp}
Let $S$ be a subset of the affine space $V$ that generates $V$ and assume that for any two points $P,Q \in S$ the line $\overline{PQ}$ is contained in $S$. Then $S=V$.
\end{Lemma}

\begin{proof}
First of all, $S$ is nonempty, since an affine space cannot be generated by the empty set. We can take a point $P_0$ in $S$ and translate everything by $P_0$ so that $P_0 = 0$, the origin, and we are reduced to showing that $S$ is a subgroup of $\F_3^n$.

Given $P \in S$, the line $\overline{P0}$ contains $-P$, so $-P \in S$. Given $P,Q \in S$, if they are distinct, the line $\overline{PQ}$ contains $-P-Q$, so $-P-Q \in S$ and, by what we previously proved, it follows that $P+Q$ is in S. If $P=Q$, then $P+Q = 2P = -P$ is also in $S$ and so we've proved that $P+Q$ is in $S$, whenever $P,Q$ are in $S$, so $S$ is a subgroup of $V$.
\end{proof}

\begin{Lemma}
There are $(3^n-1)/2$ lines passing through any given point of $\F_3^n$.
\end{Lemma}

\begin{proof}
For any fixed $P \in \F_3^n$ and any choice of $Q \ne P$ we can form the line $\ell=\overline{PQ}$ going through them. There are $3^n-1$ choices for $Q$ but, as any line has three points, for each line $\ell$, there are two choices for $Q$, so the number of lines through $P$ is $(3^n-1)/2$.
\end{proof}

\begin{Definition}
\label{def:capset}
A subset $C \subset \F_3^n$ is called a capset if it does not contain any line.
\end{Definition}

\section{Greedy constructions}

\begin{Definition}
\label{def:greedy}
A greedy construction of a capset is a procedure of the following kind. Start with $S_0 = \F_3^n$ and for $i=1,2,\ldots$ define $S_i = S_{i-1} \setminus \{P_i\}$, where the point $P_i$ is among those points $P \in S_{i-1}$ such that the number of lines through $P$ contained in $S_{i-1}$ is positive and maximal among all points of $S_{i-1}$. The procedure stops when $S_i$ is a capset. A capset generated by a greedy construction is called a greedy capset.
\end{Definition}

Since $\F_3^n$ is finite, the above procedure must terminate. If $S_i$ is not a capset, then the procedure continues to $S_{i+1}$ and it follows that the procedure always stops at a capset.

The output capset is not uniquely determined as there is a choice for $P_i$ whenever there is more than one point with the maximal count of lines. This certainly happens at the first step and can happen at subsequent ones as well.

The following result characterises greedy capsets.

\begin{Theorem}
\label{thm:main}
A subset $C \subset \F_3^n$ is a greedy capset if and only if it has the following structure. There is 
a hyperplane $H_0 \subset \F_3^n$ such that $C \cap H_0 = \emptyset$ and, if $H_1,H_2$ denote the two hyperplanes of $\F_3^n$ parallel to $H_0$, then $C \cap H_i, i=1,2$ are greedy capsets, so miss a hyperplane of $H_i$ and so on, recursively. In particular, $\# C = 2^n$.
\end{Theorem}

\begin{proof}
By induction on $n$. The base case $n=1$ follows since removing any one element of $\F_3^1$ produces a capset of $\F_3^1$ of size $2$, so both constructions are the same in this case.

To prove that a greedy capset has the claimed structure, we begin by showing the existence of the hyperplane $H_0$.

If we let $d_i$ denote the dimension of the affine space spanned by $P_1,\ldots,P_i$, then $d_i \le d_{i+1} \le d_i +1$. Since $C$ does not contain lines, $d_i=n$ if $S_i=C$, as the complement of a hyperplane contains lines. Hence, there is a value of $i$ with $d_i = n-1$. Let $i_0$ be the smallest such $i$ and $H_0$ the hyperplane spanned by $P_1,\ldots,P_{i_0}$. We claim that, for $i \ge i_0$, if $S_i \cap H_0 \ne \emptyset$, then $P_{i+1} \in H_0$.

Let's prove the claim. If $P \not\in H_0$, then the lines $\overline{PP_j}, j \le i$ are all distinct, since the $P_j$ are in $H_0$. All other lines through $P$ are contained in $S_i$. So, we need to show that there is $P \in S_i \cap H_0$ such that not all lines $\overline{PP_j}, j \le i$ are distinct. If that does not happen then $\overline{P_kP_j}, j \le k$ does not meet $S_i$, which means
that $\{P_1,\ldots,P_i\}$ satisfies the hypotheses of Lemma \ref{lem:grp} and thus is a linear space by that lemma.  This in turn implies that $\{P_1,\ldots,P_i\} = H_0$, proving the claim.

Once the procedure removes the hyperplane $H_0$, a line contained in the remaining set has to be contained in $H_1$ or $H_2$. Therefore the procedure amounts to running two simultaneous procedures of the same kind in $H_1$ and $H_2$. The claimed structure then follows by induction.

We now show that a subset $C \subset \F_3^n$ with the structure described in the statement of the theorem is a greedy capset, again by induction on $n$. To show it is a capset, consider a line $\ell$. If $\ell$ is not contained in some $H_i, i=0,1,2$, then it meets $H_i$ in exactly one point, for $i=0,1,2$ so it meets $C$ in at most two points. If $\ell$ is contained in $H_0$ it does not meet $C$, whereas if it is contained in $H_i, i=1,2$ it meets $C$ in at most two points since $C \cap H_i$ is a capset, by induction. 

To show it is greedy, note first that a line that contains two points of $H_0$, stays in $H_0$. It follows that if $P$ is a point not in $H_0$, the lines $\overline{PQ}, Q \in H_0$ are distinct. So, if we are in a stage of the construction in which we only removed points from $H_0$ to form $S_{i-1}$, this will be a greedy construction, as the number of lines contained in $S_{i-1}$ through a point of $H_0$ is at least big as the number of lines through a point not on $H_0$, which is $(3^n-1)/2 - (i-1)$. Hence, selecting a point on $H_0$ for removal is greedy. Once $H_0$ is removed, the result follows by induction on $n$.

Finally, to show that $\# C = 2^n$, let $H_0$ be a hyperplane with $C \cap H_0 = \emptyset$ and $H_1,H_2$ denote the two hyperplanes of $\F_3^n$ parallel to $H_0$, with $C \cap H_i, i=1,2$ greedy capsets. Then $\#C \cap H_i = 2^{n-1}, i=1,2$, by induction, so $\# C = 2^n$.

\end{proof}

\begin{Example} 
\label{ex:01} 
The set $C=\{0,1\}^n \subset \F_3^n$ is a greedy capset. Indeed, $C$ is obtained by first removing the hyperplane $H_0$ with equation $x_n=2$, then removing the hyperplanes with equation $x_{n-1}=2$ within the hyperplanes with equations $x_n=0, 1$ and so on.
\end{Example}

\begin{Definition}
\label{def:complete}
A capset $C \subset \F_3^n$ is a complete capset (in $\F_3^n$) if it is not contained in a larger capset of $\F_3^n$.
\end{Definition}

The capset of Example \ref{ex:01} is complete. If $R \in \F_3^n$ is not in $C$, define $P, Q$ as follows: If the $i$-th coordinate of $R$ is $2$, let the $i$-th coordinate of $P$ be $0$ and the $i$-th coordinate of $Q$ be $1$ and if the $i$-th coordinate of $R$ is $a\ne 2$, let the $i$-th coordinate of both $P,Q$ be $a$. It is then easy to show that $P,Q$ are distinct (because $R$ is not in $C$) and belong to $C$ and $P,Q,R$ are collinear. 

Not all greedy capsets are complete, as the following example shows.

\begin{Example} 
\label{ex:quadric} 
Let $S \subset \F_3^3$ be the set of solutions $(x,y,z)$ of the equation $z=x^2+y^2$. Then it can be checked directly that $S$ is a capset ($S$ is an example of an elliptic quadric) and $\#S = 9$. If we remove the origin $(0,0,0)$ from $S$ we get a capset $C$ that doesn't meet the hyperplane with equation $z=0$ and, from this, it follows easily that $C$ is a greedy capset, which is not complete by construction.
\end{Example}

It would be very interesting to understand the completions of general greedy capsets, in particular, how big they can get.



\begin{bibdiv}

\begin{biblist}



\bib{Matrix}{article}{
   author={Blasiak, Jonah},
   author={Church, Thomas},
   author={Cohn, Henry},
   author={Grochow, Joshua A.},
   author={Naslund, Eric},
   author={Sawin, William F.},
   author={Umans, Chris},
   title={On cap sets and the group-theoretic approach to matrix
   multiplication},
   journal={Discrete Anal.},
   date={2017},
   pages={Paper No. 3, 27},
}

\bib{git}{webpage}{
author={Dawson, Oliver}, 
author={Shuvaev, Oleg},
author={Voloch, Jos\'e Felipe},
title={capsets},
url={https://github.com/Chair-and-table/capsets},
}

\bib{Jordan}{article}{
   author={Ellenberg, Jordan S.},
   author={Gijswijt, Dion},
   title={On large subsets of $\Bbb F^n_q$ with no three-term arithmetic
   progression},
   journal={Ann. of Math. (2)},
   volume={185},
   date={2017},
   number={1},
   pages={339--343},
}

\bib{Nature}{article}{
author={Balog, Matej},
author={Kumar, M. Pawan},
author={Dupont, Emilien},
author={Ruiz, Francisco J. R.},
author={Ellenberg, Jordan S.},
author={Wang, Pengming},
author={Fawzi, Omar},
author={Kohli, Pushmeet},
author={Fawzi, Alhussein},
title={Mathematical discoveries from program search with large language models},
journal={Nature},
volume={625},
date={2024},
pages={468--475},
}

\bib{Tao}{webpage}{
author={Tao, Terence},
title={Open question: best bounds for cap sets},
url={https://terrytao.wordpress.com/2007/02/23/open-question-best-bounds-for-cap-sets/},
}

\bib{Tyrrell}{article}{
   author={Tyrrell, Fred},
   title={New lower bounds for cap sets},
   journal={Discrete Anal.},
   date={2023},
   pages={Paper No. 20, 18},
   review={\MR{4684862}},
}


\end{biblist}
\end{bibdiv}
\end{document}